\documentclass{article}
\usepackage{authblk}
\author[*]{Patrick Rubin-Delanchy}
\author[**]{Daniel John Lawson}
\affil[*]{Heilbronn Institute for Mathematical Research, University of Bristol, Bristol, UK}
\affil[**]{School of Social and Community Medicine and Department of Statistics, University of Bristol, Bristol, UK}
\title{Posterior predictive p-values and the convex order} 
\usepackage{amsmath, amssymb, amsthm}
\usepackage{graphicx}
\usepackage{natbib}
\usepackage[a4paper,left=3cm,right=3cm,top=3cm,bottom=3cm]{geometry}

\newtheorem{theorem}{Theorem}
\newtheorem{lemma}{Lemma}
\newtheorem{corollary}{Corollary}

\def\Prob{\mathrm{P}}
\def\mU{\mathcal{U}}

\def\mP{\mathcal{P}}

\def\mU{\mathcal{U}}

\def\rd{\mathrm{d}}
\def\N{\mathbb{N}}
\def\Z{\mathbb{Z}}
\def\R{\mathbb{R}}

\def\Ind{\mathbb{I}}

\DeclareMathOperator*{\var}{var}

\DeclareMathOperator*{\E}{E}

\def\R{\mathbb{R}}

\begin{document}
\maketitle
\begin{abstract}
Posterior predictive p-values are a common approach to Bayesian model-checking. This article analyses their frequency behaviour, that is, their distribution when the parameters and the data are drawn from the prior and the model respectively. We show that the family of possible distributions is exactly described as the distributions that are less variable than uniform on [0,1], in the convex order. In general, p-values with such a property are not conservative, and we illustrate how the theoretical worst-case error rate for false rejection can occur in practice. We describe how to correct the p-values to make them conservative in several common scenarios, for example, when interpreting a single p-value or when combining multiple p-values into an overall score of significance. We also handle the case where the p-value is estimated from posterior samples obtained from techniques such as Markov Chain or Sequential Monte Carlo. Our results place posterior predictive p-values in a much clearer theoretical framework, allowing them to be used with more assurance.
\end{abstract}

\section{Introduction}
In important papers on Bayesian model-checking, \citet{meng94} and \cite{gelman96} proposed to test the fit of a model by analysing the following posterior quantity. Let $f$ be some function measuring the discrepancy between the model and the data. The question asked is: if a new dataset were generated from the same model and parameters, what is the probability that the new discrepancy would be as large? In mathematical notation this probability is written \citep[Eq. 2.8, Eq. 7]{meng94, gelman96} 
\begin{equation}
P = \Prob\{f(D^{*}, \theta) \geq f(D, \theta) \mid D\}, \label{eq:ppp}
\end{equation}
where $\theta$ represents the model parameters, $D$ is the observed dataset, $D^*$ is a hypothetical replicated dataset generated from the model with parameters $\theta$, and $\Prob(\cdot \mid D)$ is the joint posterior distribution of $(\theta, D^*)$ given $D$. A variable $P$ of the above form is referred to as a \emph{posterior predictive p-value}. 

Since their introduction, which can be credited to \citet{guttman1967use}, \cite{rubin1984bayesianly}, \cite{meng94} or \cite{gelman96}, depending on definitions, posterior predictive p-values have received a number of criticisms. First, $P$ is a p-value, and as such its interpretation is full of pitfalls. For example, it is certainly not the probability that the model is right. Second, the dependence of $f$ on the unknown $\theta$ may seem unusual. Third, because the full posterior is used, rather than the prior \citep{box1980sampling} or a partial posterior \citep{bayarri00}, there is something self-fulfilling about this check; heuristically, one would expect $P$ to concentrate around $1/2$. 

This last issue is really part of a larger problem. At present, there is no clear mathematical description of the probabilistic behaviour of $P$, except for a few insights given in the last pages of \citet{meng94}. Over the last two decades, statements have appeared in the literature generally suggesting that the problem is `hard'. For example \citet{hjort06} say ``the interpretation and comparison of posterior predictive p-values [is] a difficult and risky matter''. \citet{bayarri00} have commented that ``Its main weakness is that there is an apparent ``double use'' of the data...This double use of the data can induce unnatural behavior''. In a discussion of \citet{gelman96}, Rubin alluded to some ``conservative operating characteristics'' \citep{rubin96}.

This article shows that the frequency behaviour of the posterior predictive p-value in \eqref{eq:ppp} is precisely described as being less variable, in the convex order, than a uniform random variable on $[0,1]$. Although the property had already been discovered by \citet[Theorem 1]{meng94}, our main contribution is that \emph{any} probability measure of this sort is the distribution of \emph{some} posterior predictive p-value (Theorem \ref{thm:main_theorem}). This leads to determining that the p-values are \emph{not conservative} in general, for example, the $2 \alpha$ bound given in \citet{meng94} is achievable (Section \ref{sec:subuniform} and Figure \ref{fig:lasso}). However, when many posterior predictive p-values are combined into an overall score, the result is sometimes conservative. For instance, we show that the product of independent and identically distributed posterior predictive p-values is stochastically larger, asymptotically, than the product of uniform variables (Fisher's method, Lemma \ref{lem:fisher_asymp}). 

A posterior predictive p-value is an informative quantity: it is the probability of the discrepancy being `as large tomorrow as it is today'. Given a sample from the posterior distribution, this probability can typically be estimated very quickly and with no difficulty. As a result, the use of this model-checking technique and its variants is widespread \citep{huelsenbeck2001bayesian, sinharay2003posterior, thornton2006approximate, steinbakk2009posterior}. These are good reasons to seek to understand the behaviour of posterior predictive p-values in repeated samples. We take no position on the philosophical validity of the approach.

In fact, understanding the behaviour of posterior predictive p-values has a more general application, that is not necessarily Bayesian. Suppose we have two random objects, $X$ and $Y$, with a known joint distribution, and only $Y$ is observed. Many common statistical models have this structure. For example, $X$ might be the underlying state in a state-space model and $Y$ the observation; or $Y$ might be a point process and $X$ an underlying random intensity, as in the Cox process \citep{daley07}. In such models we often want to test something about $Y$ based on $X$. For instance, in the standard Kalman filter model, we could be interested in testing the distance, $f(Y,X)$, between the state and the observation. Ideally we would be able to observe the true p-value, $Q = \Prob\{f(Y^*,X) \geq f(Y,X) \mid X,Y\}$, where $Y^*$ is a replicate of $Y$ conditional on the true $X$. However, this is impossible because $X$ is not observable. We therefore replace $Q$ with $P = \Prob\{f(Y^*,X) \geq f(Y,X) \mid Y\} = \E(Q \mid Y)$. Then, under the hypothesis that the model holds, how should $P$ be distributed? We will give a real example of this question arising in a cyber-security application.


Two issues that we cannot address, in the probabilistic framework that we adopt, are the following. First, we do not describe, nor even attempt to define, the frequency behaviour of $P$ if the prior on $\theta$ is improper. Second, we do not find any non-trivial \emph{lower} bound on how conservative $P$ is. This could be a matter of concern, since a high false negative rate can have particularly dangerous implications in a model-checking application. The problem is that, if $\theta,D$ and $f$ are not constrained in some way, it is possible to construct a posterior predictive p-value that is arbitrarily concentrated about $1/2$, so that a less general setup would have to be assumed. Further comments about this issue are in the Discussion.

The remainder of this article is organised as follows. Section \ref{sec:convex_order} treats the case of a single posterior predictive p-value. First, we prove our main result, that there is a posterior predictive p-value for any distribution that is less variable than uniform in the convex order, in the process also deriving an extension of a famous theorem by \citet{strassen65}. Second, we describe this family of distributions, re-proving the $2\alpha$ bound found by \citet{meng94} as a special case. Third, we construct some abstract examples of posterior predictive p-values that achieve the bound, and then present a real application in cyber-security. In Section \ref{sec:mult_test}, we treat the case of multiple posterior predictive p-values. Finally, in Section \ref{sec:estimation_error}, we compare two schemes for calculating the posterior predictive p-value from a posterior sample, both proposed by \citet{gelman96}. We show that one of the estimates, but not the other, produces a random variable that is less variable than uniform in the convex order, meaning that a number of our results continue to hold for the estimate without alteration.

\section{Main results}
\label{sec:convex_order}
We start with a joint distribution over two random elements, $\theta$ and $D$. In Bayesian statistics, this would normally be decomposed as a marginal distribution on $\theta$, called the prior, and a conditional distribution on $D \mid \theta$, called the model. For a given dataset $D$, a typical calculation of the posterior predictive p-value would proceed as follows \citep[Section 2.3]{gelman96}. First, simulate $\theta_1, \ldots, \theta_M$ from the posterior distribution of $\theta$ given $D$, for a large $M$. Second, for each $\theta_i$, simulate a replicated dataset $D^*_i$. Finally, estimate
\begin{equation}
\hat P_M = \frac{1}{M} \sum_{i=1}^M \Ind\{f(D^*_i, \theta_i) \geq f(D, \theta_i)\}, \label{eq:estimate1}
\end{equation}
where $\Ind$ is the indicator function. $P$ is the limit of $\hat P_M$ as $M \rightarrow \infty$, assuming the $\theta_i$ are independent. We will revisit the properties of the estimate under dependence and finite $M$ in Section \ref{sec:estimation_error}. For now, assume that $P$ is effectively observable for a given dataset $D$, e.g. by making $M$ large enough or through some analytical solution.

Our analysis focusses on the \emph{frequency} behaviour of $P$, meaning its behaviour when a specified joint distribution on $\theta$ and $D$ holds (heuristically, when the model and prior are right). Because $D$ is now random, $P$ is a random variable. It could be simulated as follows. To obtain a single realisation, we would draw $\theta$ from the prior, and $D$ from the model of $D \mid \theta$. Then we would discard $\theta$ and compute $P$ in \eqref{eq:ppp} conditional on $D$, e.g. via \eqref{eq:estimate1}, as if we had never seen $\theta$. To obtain multiple independent replicates of $P$, we would repeat this cycle, each time constructing a new $\theta$ and $D$.

Unless stated otherwise, the discrepancy $f(D,\theta)$ is assumed to be an absolutely continuous random variable. \citet{meng94} makes use of the identity
\begin{align*}
P &= \Prob\{f(D^{*}, \theta) \geq f(D, \theta)\mid D\}\\
& = \E[\Prob\{f(D^{*}, \theta) \geq f(D, \theta)\mid\theta,D\}\mid D],
\end{align*}
to make the following observation. For any convex function $h$, we have $\E\{h(P)\} \leq \E\{h(U)\}$, if the expectations exist, where $U$ is a uniform random variable on $[0,1]$. The proof uses the fact that the quantity $\Prob\{f(D^{*}, \theta) \geq f(D, \theta)\mid\theta,D\}$ is a random variable distributed as $U$, marginally over $\theta$ and $D$, and then applies Jensen's inequality. \citet{meng94} then finds an upper bound $\Prob[P \leq \alpha] \leq 2 \alpha$ for $\alpha \in [0,1]$. 

In fact, the property being alluded to is an important stochastic order. Let $X$ and $Y$ be two random variables with probability measures $\mu$ and $\nu$ respectively. We say that $\mu$ (respectively, $X$) is less variable than $\nu$ (respectively, $Y$) in the convex order, denoted $\mu \leq_{cx} \nu$ (or $X \leq_{cx} Y$) if, for any convex function $h$,
\[\E\{h(X)\} \leq \E\{h(Y)\},\]
whenever the expectations exist. The convex order is a statement about variability, since convex functions generally put more weight on the extremes. In fact, it has direct implications in terms of the first two moments of $X$ and $Y$. Using $h(x) = x$ and then $h(x) = -x$, two convex functions, we find $\E(X) = \E(Y)$. Then, since $\{x-\E(X)\}^2 = \{x - \E(Y)\}^2$ is a convex function in $x$, the variance of $X$ must be smaller than the variance of $Y$. In this article, we will say that a probability measure $\mP$, and a random variable distributed as $\mP$, is sub-uniform if $\mP \leq_{cx} \mU$, where $\mU$ is a uniform distribution on $[0,1]$. Posterior predictive p-values have a sub-uniform distribution.

At first glance, Meng's findings could seem quite conservative. They would suggest that, to be sure not to exceed a false positive rate of $\alpha$ when the model on $(\theta, D)$ holds, we would have to multiply our posterior predictive p-value by two. Yet, from practical experience, the variance result above, as well as a loose inspection of \eqref{eq:ppp}, we could have the impression that these p-values are already quite conservative --- even the raw p-value looks too large. This raises the question of whether the bound can be improved. More generally, it would be useful to know whether the frequency behaviour of posterior predictive p-values is well described as being sub-uniform, in other words, whether the space of distributions cannot somehow be reduced. The rest of this section addresses these questions by making the following points: 
\begin{enumerate}
\item It is possible to construct a posterior predictive p-value with any sub-uniform distribution (Theorem \ref{thm:main_theorem}).
\item Some sub-uniform distributions achieve the $2 \alpha$ bound (Corollary \ref{cor:two_alpha}).
\item Therefore, some posterior predictive p-values achieve the $2 \alpha$ bound. In fact, we can construct simple examples where this happens (Section \ref{sec:construction}). 
\end{enumerate}
This example also lends some intuition to how the problem can occur in more complicated and/or less transparent scenarios, including the real case study in Section \ref{sec:cyber}.
\subsection{A posterior predictive p-value for every sub-uniform distribution}
A famous theorem by \cite{strassen65} (see also references therein) provides a fundamental interpretation of the convex order through a martingale coupling.
\begin{theorem}[Strassen's theorem]\label{thm:strassen} 
For two probability measures $\mu$ and $\nu$ on the real line the following conditions are equivalent:
\begin{enumerate}
\item $\mu \leq_{cx} \nu$;
\item there are random variables $X$ and $Y$ with marginal distributions $\mu$ and $\nu$ respectively such that $\E (Y \mid X)=X$.
\end{enumerate}
\end{theorem}
This (simpler) version of the theorem is due to \citet{muller01}. The original version holds for more general probability measures.

Strassen's theorem is central to our main result. Given a sub-uniform probability measure $\mP$, it is possible to construct a coupling, $(P,U)$, where $P$ is distributed as $\mP$, $U$ is uniform on $[0,1]$, and $\E(U \mid P) = P$. However, to make progress, certain awkward couplings need to be forbidden, namely, those for which the conditional random variable $U \mid P$ has some discrete components. The following theorem makes this possible.

\begin{theorem}
\label{thm:continuous_coupling}
Let $\mu$ and $\nu$ be two probability measures on the real line where $\nu$ is absolutely continuous. The following conditions are equivalent:
\begin{enumerate} 
\item $\mu \leq_{cx} \nu$;
\item there exist random variables $X$ and $Y$ with marginal distributions $\mu$ and $\nu$ respectively such that $E(Y \mid X)=X$ and the random variable $Y \mid X$ is either singular, i.e. $Y=X$, or absolutely continuous with $\mu$-probability one.
\end{enumerate}
\end{theorem}

The proof is relegated to the Appendix because it is quite technical. (It may be advantangeous to first consult Section \ref{sec:subuniform} on the integrated distribution function.) On the other hand, the basic idea is simple. First, a small amount of zero-mean, continuously distributed noise is added to $X$, constructing a second variable $\tilde X$ with distribution $\tilde \mu$. The noise depends on $X$ in such a way that $\mu \leq_{cx} \tilde \mu \leq_{cx} \nu$. Second, Strassen's theorem is used to form a martingale coupling of $Y$ with $\tilde X$, i.e. $\E(Y \mid \tilde X) = \tilde X$. Then, $\E(Y \mid X)=\E\{\E(Y \mid \tilde X) \mid X\}= X$ and the details of the construction ensure that, no matter how $Y$ and $\tilde X$ are coupled, $Y \mid X$ is either continuous or singular. 

From this we are able to construct a coupling that bears more resemblance to a Bayesian model-checking setup. The following result is notably relevant to the \emph{average discrepancy} proposed by \citet{gelman96}.
\begin{lemma}\label{thm:general_version}
Let $\mu$ and $\nu$ be two probability measures on the real line where $\nu$ is absolutely continuous. The following conditions are equivalent:
\begin{enumerate}
\item $\mu \leq_{cx} \nu$; 
\item there exist real random variables $X,S,\theta$ and a collection of random variables $Y_{t\in\R}$ such that
\[X = \E(Y_{\theta}\mid S),\]
where $X$ has marginal distribution $\mu$,  $Y_{t}$ has marginal distribution $\nu$ for any $t\in \R$, and $\theta$ is independent of the collection $Y_{t\in\R}$. 
\end{enumerate}
\end{lemma}


\begin{proof}
First, we prove that (b) implies (a). Let $h$ be a convex function. Then, if the expectations exist,
\[\E\{h(X)\} = \E[h\{\E(Y_{\theta} \mid S)\}] \leq \E[\E\{h(Y_{\theta})\mid S\}] = \E\{h(Y_{t})\},\]
for an arbitrary $t \in \R$, using Jensen's inequality. Now we prove that (a) implies (b).

By Theorem \ref{thm:continuous_coupling} there exists a coupling of real random variables, $(X, S)$, such that $S \mid X$ is continuous or singular with probability one and $\E(S \mid X) = X$. Let $G$ be a continuous distribution function that is positive on $\R$. If $S \mid X$ is singular, let $Y_t=S$ for all $t \in \R$. Otherwise, $S \mid X$ has a continuous distribution function, denoted $F_{S|X}$.  If we define $Y_t$ via
\begin{equation}
Y_t=F^{-1}_{S|X}[\{F_{S|X}(S) + G(t) \}\bmod 1],\quad t \in \R,\label{eq:the_step}
\end{equation}
then $Y_t \mid X$ has the same distribution as $F^{-1}_{S|X}(U)$, where $U$ is uniformly distributed on $[0,1]$, therefore $Y_t\mid X$ is distributed as $S \mid X$ for any $X$. Hence, $Y_t$ has measure $\nu$ marginally.

Let $\theta$ be a random variable with distribution function $G$ that is independent of all previously defined random variables. If $S\mid X$ is singular then clearly $X = \E(Y_{\theta}\mid S)$. Otherwise, 
\begin{align*}
\E(Y_\theta \mid S)&=\E\left(F^{-1}_{S|X}[\{F_{S|X}(S) + G(\theta) \}\bmod 1]\right)\\
&= \E\{F^{-1}_{S|X}(U)\}=\E\{S\mid X\}=X.
\end{align*}
\end{proof}
Note that the proof is heavily reliant on the existence of a continuous coupling, guaranteed by Theorem \ref{thm:continuous_coupling}, making the step \eqref{eq:the_step} possible and essentially allowing any choice of distribution for $\theta$. We are now in a position to state our main result.

\begin{theorem}[Posterior predictive p-values and the convex order]
\label{thm:main_theorem}
$\mP$ is a sub-uniform probability measure if and only if there exist random variables $P,D,\theta$ and an absolutely continuous discrepancy $f(D,\theta)$ such that
\[P = \Prob\{f(D^*,\theta) \geq f(D,\theta) \mid D\},\]
where $P$ has measure $\mP$, $D^*$ is a replicate of $D$ conditional on $\theta$ and $\Prob(\cdot \mid D)$ is the joint posterior distribution of $(\theta,D^*)$ given $D$.
\end{theorem}

\begin{proof}
\citet[Theorem 1]{meng94} proved that $P$ is sub-uniform. To show the existence part of the proof, we now construct a coupling $\E(S \mid P) = P$ such that $S \mid P$ is continuous or singular with probability one, $P$ has marginal distribution $\mP$ and $S$ is marginally uniform on $[0,1]$. As in the proof of Lemma \ref{thm:general_version}, we arrive at a setup 
\[U_t=F^{-1}_{S|P}[\{F_{S|P}(S) + G(t) \}\bmod 1], \quad t \in \R,\]
if $S \mid P$ is continuous, and $U_t = S$ otherwise, where $G$ is some positive continuous distribution function on $\R$. 

Let $D$ be a random variable that implies $S$, i.e., there exists a function $g$ such that $S = g(D)$ with probability one, but that is otherwise independent of the other variables. Given the values of $D$ and $t$ the value of $U_t$ is known. Therefore, we can construct a discrepancy function $f$ such that $f(D,t) = \bar F^{-1}(U_t)$ with probability one, where $\bar F$ is a continuous survival function. Then, if $\theta$ has distribution $G$,
\begin{align*}
\Prob\{f(D^*,\theta) \geq f(D,\theta) \mid D\} &= \E(\Prob\{f(D^{*}, \theta) \geq f(D, \theta)\mid\theta,D\} \mid D)\\
& = \E(U_{\theta} \mid D) = \E(U_{\theta} \mid S) = P,
\end{align*}
where the last equality follows the same argument given at the end of the proof of Theorem~\ref{thm:general_version}.
\end{proof}
It is telling that the proof needs a parameter-dependent discrepancy function. It seems possible that not all sub-uniform distributions are attainable if $f$ can depend only on $D$. In fact, in his highly influential paper on Bayesian model-checking, \cite{rubin1984bayesianly} only considered p-values of this type,
\begin{equation}
\Prob\{f(D^*) \geq f(D) \mid D\}. \label{eq:rubin}
\end{equation}

It would be interesting if the frequency behaviour of this class of posterior predictive p-values turned out to be special.
\subsection{Characterising sub-uniformity}
\label{sec:subuniform}
To help explore the family of sub-uniform distributions, it will be useful to introduce the \emph{integrated} distribution function of a random variable $X$ with distribution function $F_X$,
\[\phi_X(x) = \int_{-\infty}^x F_X(t) \rd t,\]
which is defined for $x \in \R$. \citet{muller01} analysed and made extensive use of this function and its counterpart, formed from the survival function, where $(1-F_X)$ replaces $F_X$ in the above. Some of their results are restated here:
\begin{enumerate}
\item $\phi_X$ is non-decreasing and convex;
\item Its right derivative $\phi^+_X(x)$ exists and $0 \leq \phi^+_X(x) \leq 1$;
\item $\lim_{x \rightarrow  -\infty} \phi_X(x) = 0$ and $\lim_{x \rightarrow \infty} \{x - \phi_X(x)\} = \E(X)$.
\end{enumerate}
Furthermore, for any function $\phi$ satisfying these properties, there is a random variable $X$ such that $\phi$ is the integrated distribution function of $X$. The right derivative of $\phi$ is the distribution function of $X$, $F_X(x) = \phi^+(x)$.

Let $Y$ be another random variable with integrated distribution function $\phi_Y$. Then $X \leq_{cx} Y$ if and only if $\phi_X(x) \leq \phi_Y(x)$ for $x \in \R$ and $\lim_{x \rightarrow \infty} \{\phi_Y(x)-\phi_X(x)\} = 0$.

\begin{figure}[t]
  \centering
  \includegraphics[width=13cm]{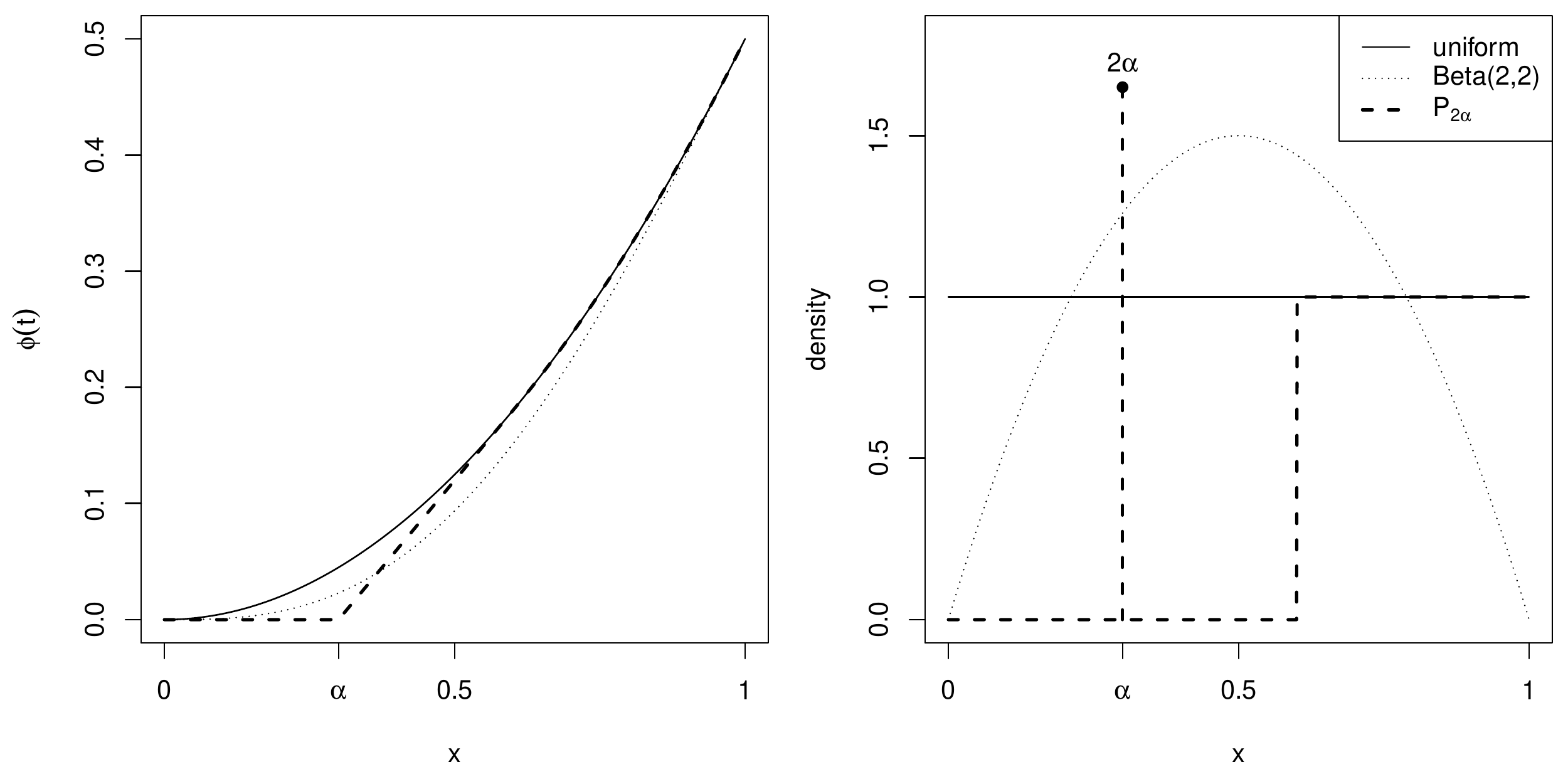}
  \caption{Examples of sub-uniform distributions. The integrated distribution function (left) and corresponding density (right) are shown for three distributions: a uniform distribution on $[0,1]$ (solid), a Beta distribution with parameters 2 and 2 (dotted) and a mixture distribution of: a point mass at $\alpha$ of probability $2\alpha$ and a uniform distribution over $[2\alpha,1]$ (dashed).}
  \label{fig:icdf}
\end{figure}
The integrated distribution function of a uniform random variable is $\phi_U(x) = x^2/2$. Figure~\ref{fig:icdf} shows this function alongside some others, corresponding to sub-uniform probability measures. From the points above, it is clear that all these functions must be non-decreasing, convex, with a right derivative between 0 and 1, always below $\phi_U(x)$, equal to $0$ at $0$ and $1/2$ at $1$. There is a one-to-one correspondence between sub-uniform probability measures and functions satisfying these criteria.

The dashed line in Figure \ref{fig:icdf} is of particular interest. It corresponds to a distribution, hereafter denoted $\mP_{2\alpha}$, which is a mixture of a point mass at $\alpha$, of probability $2\alpha$, and a uniform distribution over $[2\alpha,1]$, of probability $(1-2\alpha)$. $\mP_{2\alpha}$ is sub-uniform, as can be established by (analytically) comparing its integrated distribution function to $\phi_U$, and achieves the $2 \alpha$ bound: if $P$ is a random variable from $\mP_{2\alpha}$ then $\Prob(P \leq \alpha) = 2 \alpha$. 

This leads us to the main point of this section: Theorem~\ref{thm:main_theorem} guarantees that there is a posterior predictive p-value distributed as $\mP_{2\alpha}$, i.e., the bound of \citet{meng94} is achievable. We next provide a new insight on why the bound holds.

The following result gives a general bound on the distribution function of $X$ at a single point, given the probability measure of $Y$, when $X \leq_{cx} Y$. The proof is very simple, but we have not been able to find it elsewhere. Some related results are given by \citet{embrechts2006bounds}, on bounding the distribution function of a sum of dependent random variables with the same marginal distribution and \citet{meilijson1979convex}, on constructing a random variable $Z$, using only the distribution of $Y$, that is stochastically larger than $X$.
\begin{lemma}\label{lem:gen_2alpha}
Let $X$ and $Y$ be two random variables satisfying $X \leq_{cx} Y$, with distribution functions $F_X$ and $F_Y$ respectively. For a given $\alpha \in \R$, let
\begin{equation}
h = \min\left[1,\max\left\{w: w (x-\alpha) \leq \int_{-\infty}^x F_Y(t) \rd t, \: x \in \R\right\} \right]. \label{eq:h}
\end{equation}
Then $F_X(\alpha) \leq h$. Furthermore, there exists a random variable $\tilde X$, with distribution function $F_{\tilde X}$, such that $\tilde X \leq_{cx} Y$ and $F_{\tilde X}(\alpha) = h$.
\end{lemma}
A formal proof of this lemma is given in the Appendix, but the basic idea is illustrated in Figure~\ref{fig:icdf} with $\mP_{2\alpha}$: we find an integrated distribution function $\phi$ which has a maximal derivative at $\alpha$ subject to $\phi(x) \leq \phi_Y(x)$ for $x \in \R$ and $\lim_{x \rightarrow \infty} \{\phi_Y(x)-\phi(x)\} = 0$. For the case of a sub-uniform probability measure we find:
\begin{corollary}\label{cor:two_alpha}
Let $P$ be a sub-uniform random variable with distribution function $F_P$. Then $F_P(\alpha) \leq 2 \alpha$, for $\alpha \in [0,1/2]$. Furthermore, for any such $\alpha$, there exists a sub-uniform  random variable $\tilde P$, with distribution function $F_{\tilde P}$, satisfying $F_{\tilde P}(\alpha) = 2 \alpha$.
\end{corollary}
This corollary is only included for completeness, since everything it says is already known. The existence part of the statement is evident from $\mP_{2\alpha}$, and \citet[Eq. 5.6]{meng94} had already proved the bound.

\subsection{Two constructive examples}
\label{sec:construction}
To obtain a posterior predictive p-value that is distributed as $\mP_{2\alpha}$, rather than an arbitrary sub-uniform distribution, the structure used for the proof of Theorem \ref{thm:main_theorem} is more complicated than needed.

Let $U_0$ be a uniform random variable on $[0,1]$ and let $U_1 = U_0 \cdot \Ind(U_0 \geq 2 \alpha) + (2 \alpha-U_0) \cdot \Ind(U_0 < 2 \alpha)$. Then $(U_0 + U_1)/2$ has distribution $\mP_{2\alpha}$. This construction is due to \citet[Lemma 2]{ruschendorf1982random}. \citet{dahl2006conservativeness} found it independently and used it to form a (quite theoretical) posterior predictive p-value with distribution $\mP_{2\alpha}$. We now present two more visual examples.


Under Model 1, $X(t), t\in[0,\infty)$ denotes the position of a particle in the geometry shown in Figure \ref{fig:lasso}a as it travels from the left ($X(0)=0$), towards the loop, and then around it, either clockwise ($\theta=1$) or anti-clockwise ($\theta=0$), stopping before it has gone all the way around ($X(\infty) < 1$). The two senses of rotation are equally probable \emph{a priori}, $\Prob(\theta=0)=\Prob(\theta=1)=1/2$, and the dynamics of the particle are such that the distance travelled after one unit of time is continuously distributed, with survival function $G$, density $g$ and support on $[0,1)$.  

After one unit of time, the position of the particle is observed, $X(1)=x$, recorded going clockwise around the loop. The distance travelled along the path indexed by $\theta$ (Figure \ref{fig:lasso}b) is
$$
f(x,\theta) = \left\{
\begin{array}{llcl}
x & : x\leq 1-2\alpha, && \\
x & : x > 1-2\alpha & \mathrm{and} & \theta=1, \\
2-2\alpha-x& : x > 1-2\alpha & \mathrm{and} & \theta=0.
\end{array}
\right.
$$
The posterior probability of $\theta$ given $x$ is
\begin{align*}
p(\theta\mid x) &\propto g\{f(x,\theta)\}p(\theta),
\end{align*}
for $\theta=0,1$. We will use the distance travelled, $f$, as a discrepancy function. Let $X^*(t),t\in[0,\infty)$ be a second, hypothetical, particle in the same conditions, observed at $X^*(1) = x^*$. Given $x$ and $\theta$, the probability that the second particle would travel at least as far is $G\{f(x,\theta)\}$. Therefore, the posterior predictive p-value is
\begin{align*}
P &= \Prob\{f(x^{*}, \theta) \geq f(x, \theta) \mid x\}\\
& = \sum_{\theta=0,1} p(\theta \mid x) \Prob\{f(x^{*}, \theta) \geq f(x, \theta) \mid \theta, x\}\\
&= \sum_{\theta=0,1} p(\theta \mid x) G\{f(x,\theta)\}.
\end{align*}
If $G(t)=1-t$, $g=1$, then we cannot distinguish which direction the particle took, i.e. $p(\theta \mid x)=1/2$ for $\theta=0,1$. Then
\[P = \begin{cases} 1-x & x \leq 1-2\alpha,\\
(1-x)/2 + (x + 2\alpha -1)/2=\alpha & x > 1-2\alpha.\end{cases}\]
Now consider how $P$ would behave in repeated experiments. The observation $x$, above, is now a random variable. When $G(t)=1-t$, it is uniformly distributed on $[0,1)$, so that $P$ is distributed as $\mP_{2\alpha}$. As we vary $G$, we can construct a range of other sub-uniform distributions with $\Prob(P\leq\alpha) > \alpha$. 

Under Model 2, an observation $\mathbf{x}$ is a vector of $K$ proportions that sum to 1, or a point on the regular $(K-1)$-simplex. $\mathbf{x}$ is generated by a mixture of $K$ unimodal components, each with a mode at one distinct corner of the simplex. The components are indexed by $\theta = 1, \ldots, K$, and the corresponding corners are denoted $c_\theta$. In order to quantify the `homogeneity' of $\mathbf{x}$, we use as a discrepancy the distance between the observation and the corner corresponding to its generating component, $f(\mathbf{x},\theta)=\lVert\mathbf{x}-c_\theta\rVert$, and construct $P$ in the usual way.

Non-conservative behaviour can occur for certain parameterisations of this problem. Figure \ref{fig:lasso}c--d shows an example when the  $\mP_{2\alpha}$ distribution is achieved. This uses $K=2$ (i.e. the simplex is the unit interval), a prior $p(\theta)=1/2$ and the model
\begin{align*}
x \mid \theta=0 &\sim \mU[0,0.5+\alpha), \\
x \mid \theta=1 &\sim \mU(0.5-\alpha,1],
\end{align*}
where $\mU$ is a uniform distribution over the specified interval, $x$ is the first element of $\mathbf{x}=(x, 1-x)$ and $c_\theta=\theta$. Showing that $P$ has distribution $\mP_{2\alpha}$ proceeds analogously to Model 1. 


Model 2 is an idealisation of a real problem that is  encountered in population genetics, where the object is to identify and remove from analysis individuals with mixed genetic ancestry. The observations $\mathbf{x}$ are outputs of admixture algorithms such as STRUCTURE \citep{pritchard2000inference} and ADMIXTURE \citep{Alexander09}. Assigning  such a p-value to individuals based on their inferred admixture is one way to perform screening to create reliable reference populations.



These examples give us an intuition on how non-conservative behaviour can occur in practice. The effect comes from a) having parameter-dependent p-values $Q = \Prob\{f(D^{*}, \theta) \geq f(D, \theta)\mid\theta,D\}$ that have a conflicting view of what is `extreme' and then b) the posterior on $\theta$ not allowing a single one to dominate.
\begin{figure}[htp]
  \centering
  \includegraphics[width=15cm]{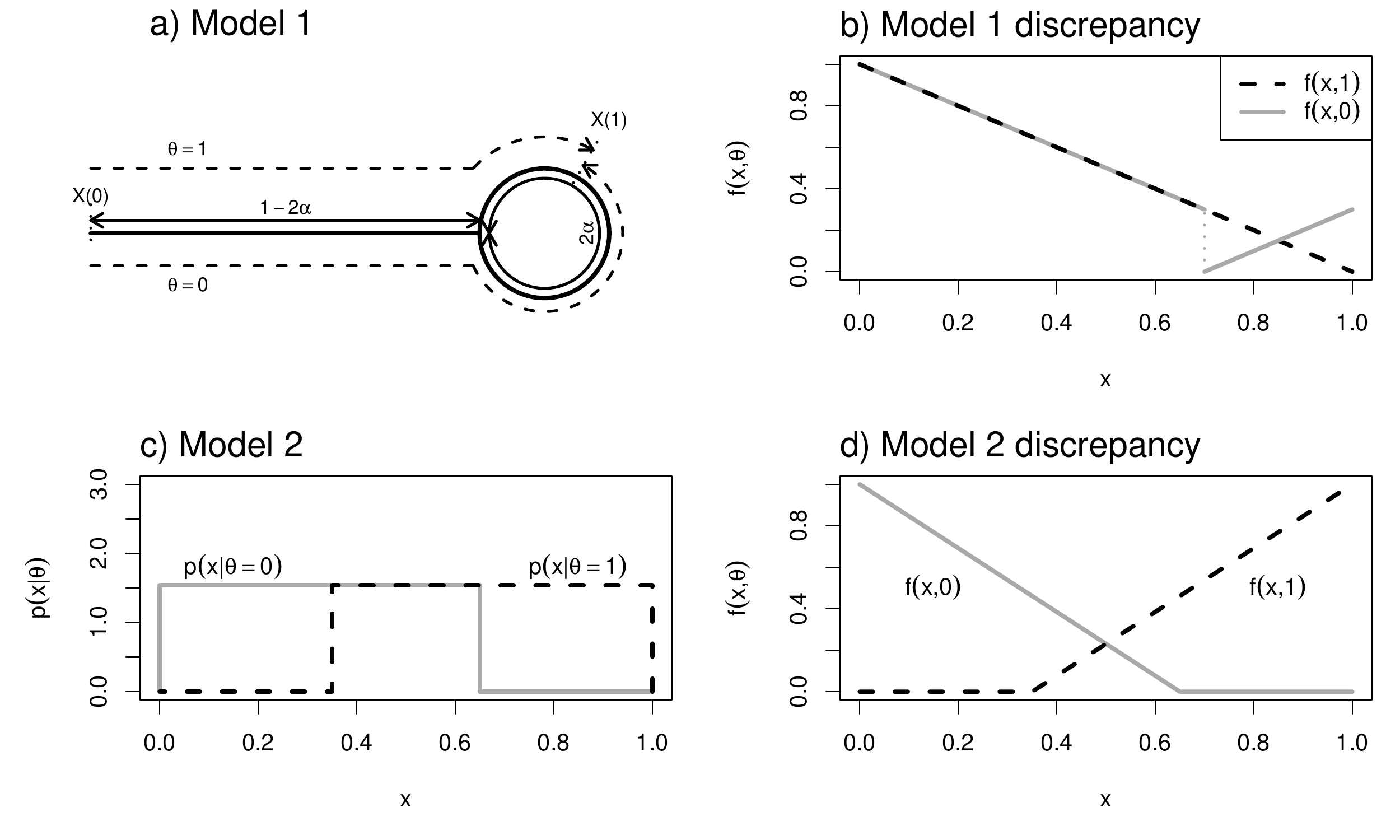}
  \caption{Two models exhibiting the $2 \alpha$ bound. a) Model 1. Consider a particle moving on the geometry shown, starting at $X(0)=0$, on the far left, and travelling either clockwise or anti-clockwise. Location is measured `clockwise' around the loop to a position $X(1)=x$. b) The discrepancy function for Model 1 is the distance travelled, measured along either of the paths $\theta=0$ or $\theta=1$. c) Model 2 density, $p(x|\theta)$, with d) the corresponding discrepancy. Both models yield a posterior predictive p-value with distribution $\mP_{2\alpha}$, see details in main text.}
  \label{fig:lasso}
\end{figure}

\subsection{A real example: intrusion detection on computer traffic}\label{sec:cyber}
There are a number of examples of the use of the posterior predictive p-value in \eqref{eq:ppp} for Bayesian model-checking, e.g. \citet{gelman96}, \citet{steinbakk2009posterior} or \cite{gelman2014bayesian}. Our interest in the problem actually stems from a different goal: anomaly detection in the presence of unknowns. The example we present is motivated by a cyber-security application, but the discussion is applicable to many problems where, loosely speaking, a test statistic is chosen on the basis of a latent parameter.

Network flow data are time-stamped records, $X_i$ say, of communications between entities on a computer network, providing limited information on the communication type and data transferred \citep{sperotto2010overview}. Modern network monitoring tools sift through these data in search for anomalies indicative of an intrusion \citep{sperotto2010overview, neil2013scan,adams2014data}. Because each record is usually generated by a single computer application, e.g., an email client or web browser, a model for these data will often include a latent parameter, $\theta_i$ say, that identifies the application that generated $X_i$. What constitutes normal and abnormal behaviour can vary substantially between applications. In testing for anomaly, therefore, it is often desirable for a test of $X_i$ to be developed on the basis of $\theta_i$. 

Within each record, there is a categorical variable describing the network protocol, referred to as the (server) port.
Ignoring a number of caveats for simplicity, this provides information about the \emph{reported} type of service a client computer is using on a server.
The well-known ports mostly fall between 0 and 1023. For example, web browsers predominantly use HTTP (80) and HTTPS (443) whilst other applications, such as Windows Update, Dropbox and file sharing tools, use a more complex range. An unusual port, given the application, could be evidence of a computer having become infected and/or engaging in covert activity. 

In what follows, the dependence on the record index, $i$, is implicit. Let $\pi$ denote the observed port, supported on $ \{0, \ldots, M\}$, reported in the record, $X$. Let $h(\cdot\: ,\theta)$ denote the probability mass function of the port used for a given application $\theta$. In practice $h$ may be learnt offline, e.g. by running different applications on a computer and observing the resulting network flow data. Conditional on $\theta$, a natural choice (and the most powerful against a uniform alternative) is to use the discrepancy function $f(\cdot\:,\theta) = -h(\cdot\:,\theta)$, i.e., report the probability of observing a port as rare as $\pi$. For known $\theta$, the p-value for $\pi$ would therefore be
\[Q = \Prob\{f(\pi^*, \theta) \geq f(\pi, \theta) \mid \theta, \pi\}= \sum_{j=1}^M h(j,\theta) \Ind\{h(j,\theta) \leq h(\pi,\theta)\},\]
which is a discrete, conservative p-value. 
Now, suppose there is a probability distribution over $\theta$ which can be interpreted as a posterior distribution on $\theta$ given $X$, denoted $p(\theta \mid X)$. This could arise from a formal Bayesian analysis or be approximated by a machine-learning classifier, e.g. Random Forests \citep{breiman2001random}. A simple means to incorporate this uncertainty is to use
\[P = \Prob\{f(\pi^*, \theta) \geq f(\pi, \theta) \mid X\} = \E(Q \mid X),\]
with the expectation taken over $\theta \mid X$. 

How should $P$ behave in normal conditions? $P$ can be conservative (aside from the issue of discreteness) if the observed value of $\pi$ strongly informs $p(\theta \mid X)$. This is the risk of a `double-use' of the data \citep{bayarri00,hjort06}. In our application, because malicious software can use an arbitary port, it would be usual (and desirable) for inference about $\theta$ given $X$ to be relatively insensitive to $\pi$.


Assuming  $p(\theta \mid X)$ is not strongly informed by $\pi$, close to uniform behaviour occurs if \emph{either} a) a single $\theta$ tends to dominate the posterior for each $X$, or
b) the probability mass functions $h(\cdot\: ,\theta)$ tend to be similar across the plausible values of $\theta$. Non-conservative behaviour occurs if there is a set of ports a) that are anomalous for all $\theta$, b) for which no $\theta$ dominates in the posterior and c) in which the ports are probability-ordered differently for different applications. 

The random variable $P$ is not sub-uniform, due to discreteness. However, we can describe $P$ by a different, but similar, stochastic order. We say that a random variable $X$ is dominated by a random variable $Y$ in the decreasing convex order, denoted $X \leq_{dcx} Y$, if, for any decreasing convex function $h$,
\[\E\{h(X)\} \leq \E\{h(Y)\},\]
whenever the expectations exist  \citep[Chapter 4]{shaked07}. We find that $P \leq_{dcx} U$,by applying the following generalisation of Theorem 1 in \citet{meng94}.
\begin{lemma}
For any measurable discrepancy function $f$, the posterior predictive p-value in \eqref{eq:ppp} satisfies
\[P \leq_{dcx} U,\]
where $U$ is a uniform random variable on $[0,1]$.
\end{lemma}
\begin{proof}
Let $Q = \Prob\{f(D^{*}, \theta) \geq f(D, \theta)\mid\theta,D\}$. Then $Q \geq_{st} U$, where $\geq_{st}$ denotes the usual stochastic order \citep{shaked07}. Therefore, there exists a random variable $\tilde Q$, on the same probability space as $Q$, that has a uniform distribution marginally and satisfies $\tilde Q \leq Q$ with probability one \citep[Theorem 1.A.1]{shaked07}. Then, for any decreasing convex function $h$,
\[\E\{h(P)\} = \E[h\{E(Q \mid D)\}] \leq \E[h\{\E(\tilde Q \mid D)\}] \leq \E[\E\{h(\tilde Q)\mid D\}] = \E\{h(U)\},\]
using Jensen's inequality.
\end{proof}
As a particular application of this result, the proofs of Lemma~\ref{lem:gen_2alpha} and Corollary~\ref{cor:two_alpha} can be modified to show that $\Prob(P \leq \alpha) \leq 2 \alpha$.

\section{Multiple testing}
\label{sec:mult_test}
A consequence of our findings is that, for the first time, it is possible to address the treatment of multiple posterior predictive p-values formally. Suppose we have discrepancy functions, $f_1, \ldots, f_m$, giving posterior predictive p-values $P_1, \ldots, P_m$ respectively, that are to be combined into one, overall, anomaly score. A conservative solution would be to multiply every p-value by two before any analysis. This section investigates potential improvements.

The $P_i$ could occur from testing a few specific hypotheses, or from more generic bulk testing of the data, in which case we might obtain, for example, a p-value for every observation. These different scenarios affect whether the p-values can be treated as independent and/or identically distributed (under the null hypothesis that the model holds) and, also, what order of magnitude we might expect for $m$. In the analysis below, the $P_i$ are always assumed to be at least independent. 

Fisher's method \citep{Fisher48} is a popular way of combining p-values. Suppose we have classical p-values, $U_1, \ldots, U_m$, which are independent uniform random variables on $[0,1]$ under the null hypothesis. Then the statistic $-2\sum \log(U_i)$, called Fisher's score, has a $\chi^2$ distribution with $2m$ degrees of freedom. The null hypothesis is rejected when this statistic is large. Replacing the $U_i$ with $P_i$ in this procedure has an interesting asymptotic effect:
\begin{lemma}[Fisher's method is asymptotically conservative]\label{lem:fisher_asymp}
Let $P_1, \ldots, P_m$ and $U_1, \ldots, U_m$ each be sequences of independent and identically distributed sub-uniform and uniform random variables on $[0,1]$ respectively. For $\alpha \in (0,1]$, let $t_{\alpha,m}$ be the critical value defined by
\[\Prob\left(- 2 \sum_{i=1}^m \log(U_i) \geq t_{\alpha,m}\right) = \alpha.\]
Then there exists  $n \in \N$ such that 
\[\Prob\left(- 2 \sum_{i=1}^m \log(P_i) \geq t_{\alpha, m}\right) \leq \alpha,\]
for any $m \geq n$.
\end{lemma}
Hence, we can dispense with the conservative correction entirely if $m$ is large enough and the $P_i$ are identically distributed. A formal proof is given in the Appendix. Since $\E\{-\log(P_i)\} \leq \E\{-\log(U_i)\}$, from the definition of the convex order, a direct application of the law of large numbers gets us most of way, except for the possibility $\E\{-\log(P_i)\} = \E\{-\log(U_i)\}$. In fact, this exception is no problem because, perhaps surprisingly, it implies that the $P_i$ are uniform, see \citet[Theorem 3.A.43]{shaked07} or Lemma \ref{lem:mean_and_var} in the Appendix. 

In the finite, non-identically distributed case, we were able to derive three probability bounds. None beats the other two uniformly for all $m$ and all significance levels (see Figure \ref{fig:fisher}), but of course in practice the minimum can be used.

\begin{lemma}\label{lem:fisher_finite}
Let $P_1, \ldots, P_m$ be a sequence of independent sub-uniform random variables. Then for $x \geq 2 m$, 
\begin{multline*}
\Prob\left(-2 \sum_{i=1}^m \log(P_i) \geq x\right) \leq \min\Big[S_{2m}(x - 2 m \log 2), \\
m\left/\left[m+\{(x-2m)/2\}^2\right]\right., \exp\{m-x/2-m\log(2 m/x)\}\Big],
\end{multline*}
where $S_k$ is the survival function of a $\chi^2$ variable with $k$ degrees of freedom.
\end{lemma}
The first uses the $2\alpha$ bound directly (Corollary \ref{cor:two_alpha}). The second uses bounds on the mean and variance of $-\log(P_i)$ (given in Lemma \ref{lem:mean_and_var}, in the Appendix) and then applies the Chebyshev-Cantelli inequality. The third is based on a bound on the moment generating function of $-\log(P_i)$. The derivation details are in the Appendix.

\begin{figure}[t]
  \centering
  \includegraphics[width=13cm]{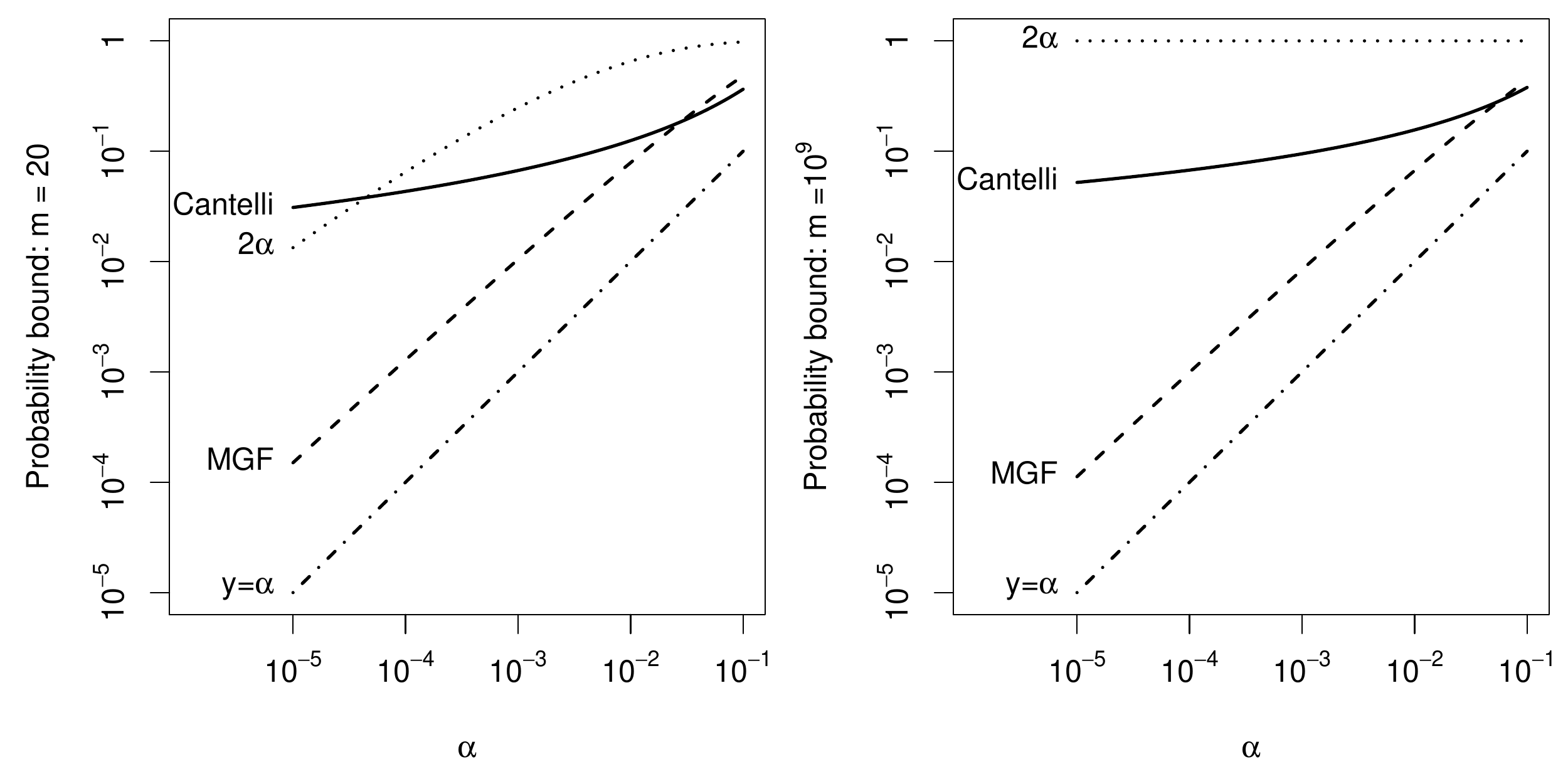}
  \caption{Comparison of the probability bounds for different nominal Fisher scores. Lemma~\ref{lem:fisher_finite} gives explicit formulae for $2 \alpha$, Cantelli and MGF, in that order. The line $y=\alpha$ provides the nominal score, i.e., the upper tail probability of the Fisher score when the component p-values are uniformly distributed. Both axes are on the logarithmic scale.}
  \label{fig:fisher}
\end{figure}

Figure \ref{fig:fisher} presents the behaviour of the different bounds under different conditions. For a given $m$ (20 on the left and 1 billion on the right) and $\alpha$, we compute the critical value, $t_{\alpha,m}$. The curves show the bound given by each formula, i.e. inputting $x = t_{\alpha, m}$ in Lemma~\ref{lem:fisher_finite}, as $\alpha$ ranges from $10^{-5}$ to $0.1$. For low $\alpha$, the bound based on the moment generating function, marked MGF, is by far the best. For large $m$, the bound based on multiplying every p-value by two, essentially the method we are trying to beat, performs very poorly. 

Rather than combine the p-values, it may be of preliminary interest to investigate just the most significant p-value, $\min(P_i)$. We may want to recalibrate this statistic to account for multiple testing. Here, to be conservative, it turns out that we cannot improve over doubling every p-value before recalibrating. This is shown in the next lemma.

\begin{lemma}
Let $P_1, \ldots, P_m$ and $U_1, \ldots, U_m$ each be sequences of independent sub-uniform and uniform random variables on $[0,1]$, respectively. For $x \in [0,1]$, let  $q = 1- (1-x)^m = \Prob\{\min(U_i)\leq x\}$. Then
\begin{align*}
\Prob\{\min(P_i) \leq x\} &\leq 1-(1-2x)^m,\\
&= 1-\{2(1-q)^{1/m}-1\}^m,
\end{align*}
which is no larger than $2 q$ and tends to $2q - q^2$ as $m \rightarrow \infty$. Furthermore, this bound is achievable if the $P_i$ are independent and identically distributed.
\end{lemma}

\begin{proof}
Let $F_i$ denote the distribution function of $P_i$. Then
\begin{align*}
\Prob\{\min(P_i) \leq x\} &= 1-\prod \{1-F_i(x)\}\\
&\leq 1 - (1- 2x)^m\\
& = 1-\{2(1-q)^{1/m}-1\}^m \\
& = 1 - (1-q) (2 - (1-q)^{-1/m})^m\\
& \rightarrow 1 - (1-q)^2 = 2 q -q^2
\end{align*}
using Corollary \ref{cor:two_alpha} in the second line (and the fact that the bound is achievable), and the formulae $(1+a/m)^m \rightarrow \exp(a)$ and $m (a^{1/m} -1) \rightarrow \log(a)$ in the fifth line. The expression $\{2(1-q)^{1/m}-1\}^m$ is an increasing function of $(1-q)^{1/m}$, which is itself increasing in $m$, therefore the composition is increasing. Hence, $1-\{2(1-q)^{1/m}-1\}^m$ attains its maximum at $m=1$, where it is $2q$.
\end{proof}
We do not pursue the topic of multiple testing any further, but clearly there is scope for further research in this direction. 
\section{Estimation schemes}
\label{sec:estimation_error}
In practice, the posterior predictive p-value will often be estimated by simulation. We now characterise the distribution of the estimate. 
Assume that, for any $D$, we can sample a sequence $\theta_1, \ldots, \theta_M$, that may or may not be dependent, from the posterior distribution of $\theta$ given $D$. Furthermore, for any $\theta_i$, we can simulate a replicate dataset $D^*_i$ independently. These are fairly usual conditions. A typical reason for the $\theta_i$ to be dependent is for them to have been generated by a Markov chain Monte Carlo algorithm. 

Suppose $M=1$ in \eqref{eq:estimate1}. $P$ is estimated from one indicator, $\hat P_1 = \Ind\{f(D^*_1, \theta_1) \geq f(D, \theta_1)\}$. Since $f(D^*, \theta_1)$ and $f(D, \theta_1)$ are identically distributed, marginally, $\hat P_1$ is a Bernoulli random variable with success probability 1/2 (remember $f(D,\theta)$ is absolutely continuous). This not a sub-uniform random variable; in fact, with respect to the convex order,  $\hat P_1$ is the maximal random variable that has mean $1/2$ and support on $[0,1]$ \citep[Theorem 3.A.24]{shaked07}. Although the point is somewhat pedantic, for any fixed and finite $M$ the calculation \eqref{eq:estimate1} will usually return identically zero or one with some positive probability, so that the estimate can rarely be sub-uniform. 

Instead, suppose it is possible to compute $\Prob\{f(D^*_i, \theta_i) \geq f(D, \theta_i) \mid \theta_i,D\}$, for any $\theta_i$ and $D$, and consider the alternative estimate, also mentioned in \citet[Section 2.3]{gelman96},
\begin{equation}
\hat R_M = \frac{1}{M} \sum_{i=1}^M \Prob\{f(D^*_i, \theta_i) \geq f(D, \theta_i) \mid \theta_i,D\}.\label{eq:estimate2}
\end{equation}
Intuitively, this estimate should do better because it is as if an infinite number of draws of $D^*_i$ were made for every $\theta_i$. Again, consider the case $M=1$. Viewed over the joint distribution of $\theta$ and $D$, the variable $\hat R_1$ is a uniform random variable over $[0,1]$ (compare to $\hat P_1$ which was Bernoulli). To see this, first note that the random variable $\Prob\{f(D^*_i, \theta_i) \geq f(D, \theta_i) \mid \theta_i,D\}$ is distributionally identical to $\Prob\{f(D^*, \theta) \geq f(D, \theta) \mid \theta,D\}=Q$, say. Then the conditional random variable $Q \mid \theta$ is uniform (for the same reason any classical p-value is uniform). Therefore $Q$ is also uniform marginally.

The estimate $\hat R_M$ is an average of uniform random variables which, regardless of any dependence, must be sub-uniform \citep[Theorem 3.A.36]{shaked07}. Therefore, remarkably, much of the stochastic behaviour of $\hat R_M$ can also be understood by the methods of this article. We have shown:
\begin{lemma}
Let $f$ be a function of $D$ and $\theta$, which in turn have a joint distribution such that $f(D,\theta)$ is an absolutely continuous random variable. For a fixed $M$, let $\theta_1, \ldots, \theta_M$ be replicates of $\theta$ given $D$, with arbitrary dependence, and let $D^*_i$ be an independent replicate of $D$ given $\theta_i$, for $i=1, \ldots, M$. Then the estimate $\hat R_M$, defined in \eqref{eq:estimate2}, is sub-uniform. In particular, $\Prob(\hat R_M \leq \alpha) \leq 2 \alpha$, for $\alpha \in [0,1/2]$.
\end{lemma}

\section{Discussion}
We have shown that the family of distributions that are less variable than uniform on $[0,1]$, in the convex order, fully characterises the frequency behaviour of posterior predictive p-values. From the properties of this order we established various probability bounds that can be used for conservative testing. Most of the resulting recommendations are straightforward, e.g., multiply the p-value by two or, Fisher's method is asymptotically conservative.

There are other approaches to Bayesian model-checking, such as partial \citep{bayarri00} or recalibrated \citep{hjort06} predictive p-values, which circumvent any need for bounds by creating a perfectly uniform statistic. Of course these methods have their own problems (mostly an implementation and computational burden) but they do address an issue that remains largely unsolved in this article, which is that for everyday models and data, posterior predictive p-values do seem to be very conservative. 

A feature we have observed is that this is certainly true with relatively simple models. However, we anticipate that in more structured, complex models the full spectrum of sub-uniform distributions could occur. In particular, `robust' models, for which parameter estimates become \emph{less} certain as the data become more anomalous, are likely to generate posterior predictive p-values with non-conservative characteristics.

That being said, one of the key objectives in the future has to be to find simply identifiable sub-classes of models and tests for which our bounds can be reduced. For example, we conjecture that the p-values of \citet{rubin1984bayesianly}, Equation \eqref{eq:rubin}, which do not allow the test to depend on the parameter, can be bounded differently.

\section*{Acknowledgements}
PRD is funded by the Heilbronn Institute for Mathematical Research. DJL is funded by the Wellcome Trust and Royal Society on Grant Number WT104125AIA.
\appendix
\section*{Appendix}
\begin{proof}[Theorem \ref{thm:continuous_coupling}]
It is straightforward to prove (and already known) that the existence of the martingale representation implies the convex order, by Jensen's inequality. Here we focus on the converse statement.  We will rely on the properties of integrated distribution functions, given at the beginning of Section \ref{sec:subuniform}. 

Let $\phi_X$ and $\phi_Y$ be the integrated distribution functions of $\mu$ and $\nu$ respectively, so that $\phi_X(x) \leq \phi_Y(x)$ for $x \in \R$. If $\phi_X(x) = \phi_Y(x)$ for all $x$ then let $Y \mid X = X$ and the proof is finished. Otherwise, because both functions are continuous the set $\{x \in \R: \varphi_X (x) < \varphi_Y(x)\}$ can be partitioned into a countable set of open intervals $C_i, i\in\N$. Consider one such interval, $C=(a,b)$ (allowing $a=-\infty$ and $b = \infty$). First we show that it is possible to construct a linear interpolation of $\phi_X$ over $C$, denoted $\phi_X^*$, at a set of points of $\mu$-measure 0 such that $\phi^*_X(x) \leq \phi_Y(x)$ for $x \in C$. Choose a point $x_0 \in C$ of $\mu$-measure 0 and fix some $\beta \in (0,1)$. 
We construct the interpolating points $x_j, j\in\Z$ recursively from $x_0$. We show how to construct $x_1$ from $x_0$, then $x_2$ from $x_1$ and so on. The interpolating points $x_{-1}, x_{-2}, \ldots$ are created similarly. For $j \in \N$ let 
\begin{multline*}
x_{j+1}' = \sup\{x \in [x_{j}, b): \forall \alpha \in [0,1]: \alpha \phi_X(x_{j}) + (1-\alpha)\phi_X(x) \leq \phi_Y[\alpha x_{j} + (1-\alpha) x]\}.
\end{multline*}
If $x'_{j+1} = \infty$, which is only possible if $b = \infty$, let $x_{j+1} = x'_{j+1} = b = \infty$. Otherwise choose $x_{j+1}$ to be a point in $[x_{j+1}'-\beta(x_{j+1}'-x_{j}), x_{j+1}']$ such that $\mu(\{x_{j+1}\}) = 0$. Stop the procedure if $x_{j+1}=b$. We claim that for any $x \in [x_0,b)$, $\sup(j \in \N_0: x_j \leq x) < \infty$. Otherwise, for any $\epsilon > 0$ there would exist $j \in \N$ such that $x_{j+1}'-x_{j} \leq \epsilon$ and a solution for $\alpha$ to 
$\alpha \phi_X(x_{j}) + (1-\alpha)\phi_X(y) \geq \phi_Y[\alpha x_{j} + (1-\alpha) y]$, where $y = \min(x_{j} + 2\epsilon, x)$. Then $\phi_Y(x_j) \leq \phi_X(y) \leq \phi_X(x_j) + (y - x_j)$, first using the fact that both $\phi_X$ and $\phi_Y$ are non-decreasing and then using $\phi_X^+ \leq 1$. This implies $\phi_Y(x_{j})-\phi_X(x_{j}) \leq 2 \epsilon$.  Therefore the functions $\phi_X$ and $\phi_Y$ would come arbitrarily close to each other over the closed interval $[x_0, x]$. Since both are continuous, by the extreme value theorem we would have $\phi_X(z) = \phi_Y(z)$ for some $z\in[x_0,x]$, which is impossible since $z \in C$. 

By a similar construction we form $x_{-1}, x_{-2}, \ldots$ The set of all intervals $(x_j, x_{j+1})$ constructed for every $C_i$ is countable. 
Denote these by $I_n=(l_n, u_n), n\in\N$, let $S = \R \setminus (\cup I_n)$ and finally define the Markov kernel from $\R$ onto $\R$,
\[K(x,\rd y) = \begin{cases} \delta_x & x \in S,\\
k_n(x,\rd y) & x \in I_n,\end{cases}\]
where $\delta_x$ denotes the point mass at $x$, and $k_n(x,\rd y)$ is a Markov kernel with the following properties. For every $x \in I_n$, $k_n(x, \rd y)$ is absolutely continuous, $\int_{I_n} k_n(x, \rd y) = 1$ and $\int_{I_n} y k_n(x, \rd y) = x$. Furthermore, for any measurable set $A \subseteq I_n$ such that 
\[\int_{I_n} k_n(x, A) \mu(\rd x) = 0,\] 
there is no $p$ in the support of $\mu$ such that $k_n(p, A)>0$. 

An example of an admissible choice for $k_n$ would be for $k_n(x, \rd y)$ to be a uniform distribution over the interval centered at $x$ with length $2\min(x-l_n, u_n-x)$. To see this, suppose that $k_n(p, A) = v > 0$ for $p$ in the support of $\mu$. It is clear from our choice of $k_n$ that there is an open neighbourhood $N$ of $p$ for which $\sup_E|k_n(p, E) - k_n(x,E)| \leq v/2$ for any $x \in N$, the supremum taken over sets $E$ in the $\sigma$-algebra of $\mu$. Therefore, 
\begin{align*} 
\int_{I_n} k_n(x, A) \mu(\rd x) & \geq \int_{N} k_n(x, A) \mu(\rd x)\\
& \geq v/2 \int_N \mu(\rd x)> 0.
\end{align*}
The last inequality comes from $N$ being an open neighbourhood of a supported point.

Let $\tilde X$ be the random variable that results from applying $K$ to $X$. We now show $\tilde \mu \leq_{cx} \nu$ where $\tilde \mu$ is the probability measure of $\tilde X$. For any $x \in S$ the kernel does not allow movement from the right to the left or the left to the right of $x$. Therefore, for $x \in S$, 
\begin{align*}
\phi_{\tilde X}(x) &= \E\{(x - \tilde X)_+\} \\
&=\Prob(\tilde X \leq  x) \E(x - \tilde X\mid \tilde X \leq x)\\
&= \E\{(x - X)_+\} = \phi_X(x),
\end{align*}
using the fact that $K$ is mean-preserving. For $x \in I_n$ the convexity of $\phi_{\tilde X}$ implies 
\begin{align*}\phi_{\tilde X}(x) & \leq \frac{(x-l_n)}{(u_n - l_n)} \phi_{\tilde X}(l_n) + \frac{(u_n-x)}{(u_n - l_n)} \phi_{\tilde X}(u_n) \\
&= \frac{(x-l_n)}{(u_n - l_n)} \phi_{X}(l_n) + \frac{(u_n-x)}{(u_n - l_n)} \phi_{X}(u_n) \leq \phi_Y(x),
\end{align*}
using $l_n, u_n \in S$ for the equality and the construction of $l_n$ and $u_n$ for the second inequality. Since $\tilde \mu$ has the same expectation as $\mu$ and therefore $\nu$, we conclude $\tilde \mu \leq_{cx} \nu$ \citep[p. 110]{shaked07}.

Finally, by Strassen's theorem there exists a random variable $Y$ with probability measure $\nu$ such that $\E (Y \mid \tilde X) = \tilde X$. This random variable satisfies $\E (Y \mid X) = X$. For $p \in S$, the random variable $Y \mid X=p$ is singular except potentially at the set of interpolating points $\{x_n, n\in\N\}$ which was constructed to have $\mu$-measure $0$. Suppose there exists a supported point $p \in I_n$ and $q \in \R$ such that $\E(\delta_q \mid  X = p) > 0$. Then since 
\[\E (\delta_q \mid X = p) = \int \E(\delta_q \mid \tilde X = y) k_n(p, \rd y),\]
there must exist a set $A \in I_n$ such that $k_n(p, A) > 0$ and $\E(\delta_q| \tilde X \in A)>0$. Since $Y$ is absolutely continuous we also have $\tilde \mu(A) = 0 = \int_{I_n} k_n(x,A) \mu(\rd x)$ violating the construction of $k_n$. Hence there are no supported points in $\cup I_n$, and only potentially a set of $\mu$-measure $0$ in $S$, such that $Y \mid X$ is neither singular nor absolutely continuous.
\end{proof}

\begin{proof}[Lemma \ref{lem:gen_2alpha}]
Let $\phi_X$ and $\phi_Y$ denote the integrated distribution functions of $X$ and $Y$ respectively. The function $\phi_X$ is non-negative, continuous and convex, therefore the set $\{w: w (x-\alpha) \leq \phi_Y(x), \: x \in \R\}$ is non-empty (it contains $0$) and closed. Hence, the maximum in \eqref{eq:h} is well-defined. For $x \in \R$ we have 
\begin{align*}
F_X(\alpha)(x-\alpha) &= \phi_X^+(\alpha) (x-\alpha) \\
& \leq \phi_X(\alpha) + (x-\alpha) \phi_X^+(\alpha)\\
& \leq \phi_X(x) \leq \phi_Y(x),
\end{align*}
using the non-negativity and convexity of $\phi_X$. Hence, $F_X(\alpha) \leq h$. If $h = 1$ then the singular random variable $\tilde X = \E(Y)$ satisfies $\tilde X \leq_{cx} Y$ and $F_{\tilde X}(\alpha) = h$. Otherwise, the set $\{x \in \R: h (x-\alpha) \leq \phi_Y(x)\}$ is closed and non-empty, again containing $0$. Therefore $\beta = \max\{x: h (x-\alpha) \leq \phi_Y(x)\}$ is well-defined and satisfies $h (\beta - \alpha) = \phi_Y(\beta)$. If $h \leq 1$, consider 
\[
\phi(x) = \begin{cases} 0 & x \leq \alpha,\\
h (x - \alpha) & \alpha \leq x \leq \beta,\\
\phi_Y(x) & x \geq \beta. \end{cases}
\]
This is a valid integrated distribution function, in particular, it is convex because $\phi^+_Y(\beta) \geq h$ (otherwise $\phi_Y$ and $h(x-\alpha)$ would cross). Moreover, $\phi \leq \phi_Y$ and $\lim_{x \rightarrow \infty} \{\phi(x)-\phi_Y(x)\} = 0$. Let $\tilde X$ be a random variable with integrated distribution function $\phi$. Then $\tilde X \leq_{cx} Y$, and $F_{\tilde X}(\alpha) = \phi^+(\alpha) = h$. 
\end{proof}
The proofs of Lemmas \ref{lem:fisher_asymp} and \ref{lem:fisher_finite} both need the following result. 
\begin{lemma}\label{lem:mean_and_var}
Let $P$ be a sub-uniform probability measure. Then either i) $P$ is uniform on $[0,1]$  or ii)
\[\E\{-\log(P)\} < \E\{-\log(U)\} = 1; \quad \var\{-\log(P)\} < \var\{-\log(U)\} = 1, \]
where $U$ is a uniform random variable on $[0,1]$
\end{lemma}

\begin{proof}
\citet[Theorem 3.A.43]{shaked07} provide the following theorem. If $X \leq_{cx} Y$ and for some strictly convex function $h$ we have $\E\{h(X)\} = \E\{h(Y)\}$ then $X$ is distributed as $Y$. The function $-\log(x)$ is strictly convex, therefore either $P$ is uniform or $\E\{-\log(P)\} < \E\{-\log(U)\}$. If the latter is true, then
\begin{align*}
\var\{-\log(P)\} &= \E[-\log(P) - \E\{-\log(P)\}]^2 \\
&< \E[-\log(P) - \E\{-\log(U)\}]^2,\\
&\leq \E\{\log(U) + 1\}^2\\
&=\var\{-\log(U)\}
\end{align*}
In the second line, the fact that the expected squared distance from the mean is smaller than from any other point is used, and in the fourth we used the fact that $(\log(x)+1)^2$ is convex.
\end{proof}

\begin{proof}[Lemma \ref{lem:fisher_asymp}]
Let $X_i = -2\log(P_i)$, $\mu_X = \E(X_i)$, $Y_i = -2 \log(U_i)$, and $\mu_Y=\E(Y_i)$. If $\mu_X = \mu_Y$ then by Lemma \ref{lem:mean_and_var} the $P_i$ are uniform on $[0,1]$ and we are done. The statement is also true if $\alpha=1$. Therefore assume $\mu_X < \mu_Y$, $\alpha \in (0,1)$ and let $t \in (\mu_X, \mu_Y)$. By the weak law of large numbers there exists an $n' \in \N$ such that, for $m \geq n'$,

\[\Prob\left(\sum_{i=1}^m Y_i \geq m t\right) \geq \alpha,\]
so that $t_{\alpha,m} \geq mt$. Therefore, for $m \geq n'$,
\begin{align*}
\Prob\left(\sum_{i=1}^m X_i \geq t_{\alpha,m}\right) \leq \Prob\left(\sum_{i=1}^m X_i \geq m t\right).
\end{align*}
Again by the law of large numbers, the right-hand side tends to zero. Hence there exists an $n \geq n'$ such that it is bounded by $\alpha$ for $m \geq n$.
\end{proof}

\begin{proof}[Lemma \ref{lem:fisher_finite}]
Let $R_m = -2 \sum \log(P_i)$. From Corollary \ref{cor:two_alpha}, we have $U_i/2 \leq_{st} P_i$, for $i = 1, \ldots, m$, where $U_1, \ldots, U_m$ are independent uniform random variables on $[0,1]$ and  $\leq_{st}$ denotes the usual stochastic order \citep[Chapter A.1]{shaked07}. This implies $-\log(P_i) \leq_{st} -\log(U_i/2)$. Because the usual stochastic order is closed under convolution \citep[Theorem 1.A.3]{shaked07}, we have $R_m \leq_{st} -2 \sum \log(U_i) + 2 m \log 2$. The sum $-2 \sum \log(U_i)$ has a $\chi^2$ distribution with $2 m$ degrees of freedom, proving the first bound. Lemma \ref{lem:mean_and_var} implies $\E(R_m) \leq 2m$ and $\var(R_m) \leq 4m$. Therefore, using Cantelli's inequality,
\begin{align*}
\Prob[R_m \geq x] & \leq \var(R_m)/\left[\var(R_m) + \{x-\E(R_m)\}^2\right]\\
& \leq \var(R_m)/\left[\var(R_m) + \{x-2m\}^2\right]\\
& \leq m/\left[m + \{(x-2m)/2\}^2\right],
\end{align*}
for $x \geq 2 m$. This proves the second bound. Finally, the moment generating function of $R_m$ is $\E\{\exp(t R_m)\} = \prod \E(P_i^{-2t})$ for $t \geq 0$. For $t \in [0,1/2)$ each $\E(P_i^{-2t}) \leq \E(U^{-2t}) = (1-2t)^{-1}$ since $x^{-2t}$ is a convex function in $x$ for $x \in [0,1]$. 
Using Markov's inequality,
\begin{align*}
\Prob(R_m\geq x) &= \Prob\{\exp(t R_m) \geq \exp(t x)\}\\
& \leq \exp(-t x)\E\{\exp(t R_m)\} \\
& \leq \exp(-t x - m \log(1-2t)),
\end{align*}
for $t \in [0,1/2)$. The minimum of this function is at $t = 1/2-m/x$, giving the third bound.
\end{proof}
\bibliographystyle{apalike} 
\bibliography{pppvalues}
\end{document}